\documentclass[reqno,11pt]{amsart}
\usepackage[foot]{amsaddr}
\usepackage{bbold}
\usepackage{charter}
\usepackage{enumitem}
\usepackage[margin=1in]{geometry}
\usepackage[colorlinks=true,linktoc=all,linkcolor=blue,citecolor=blue]{hyperref}
\usepackage{pgfplots}

\pgfplotsset{compat=newest,tick label style={font=\scriptsize}}
\setlist[enumerate,1]{label = \upshape(\roman*), ref = (\arabic*)}

\theoremstyle{plain}
\newtheorem{theorem}{Theorem}[section]
\newtheorem{proposition}[theorem]{Proposition}
\newtheorem{lemma}[theorem]{Lemma}
\newtheorem{corollary}[theorem]{Corollary}
\theoremstyle{definition}

\newtheorem{remark}[theorem]{Remark}
\newtheorem{example}[theorem]{Example}

\numberwithin{equation}{section}
\everymath{\displaystyle}
\newtheorem{caseinner}{Case}
\newenvironment{case}[1]{%
  \renewcommand\thecaseinner{#1}%
  \caseinner
}{\endcaseinner}

\newcommand{\myeqref}[1]{Eqn.~\eqref{#1}}
\newcommand{\D}{\mathbb{D}}
\newcommand{\N}{\mathbb{N}}
\newcommand{\C}{\mathbb{C}}
\newcommand{\DI}{\mathcal{D}}
\newcommand{\CO}{C_\varphi}
\newcommand{\CDO}{D_\varphi}
\newcommand{\Ph}{\varphi}
\newcommand{\SI}{\sigma}
\newcommand{\bigchi}{\mbox{\Large$\chi$}}

\title{Composition-Differentiation Operators on the Dirichlet Space}

\author{Robert F.~Allen\textsuperscript{1}, Katherine C.~Heller\textsuperscript{2}, and Matthew A.~Pons\textsuperscript{2}}
\address{\textsuperscript{1}Department of Mathematics \& Statistics, University of Wisconsin-La Crosse}
\address{\textsuperscript{2}Department of Mathematics \& Actuarial Science, North Central College}
\email{rallen@@uwlax.edu, kheller@noctrl.edu, mapons@noctrl.edu}

\begin{document}

\begin{abstract} 
We investigate composition-differentiation operators acting on the Dirichlet space of the unit disk.   Specifically, we determine characterizations for bounded, compact, and Hilbert-Schmidt composition-differentiation operators.  In addition, for particular classes of inducing maps, we derive an adjoint formula, compute the norm, and identify the spectrum.
\end{abstract}

\maketitle

\section{Introduction}

For a space $\mathcal{X}$ of functions analytic on the unit disk and for a self-map  $\Ph$ of the unit disk, we define the composition operator $\CO$ on $\mathcal{X}$ by $\CO(f)=f\circ\Ph$. The study of such operators has been ongoing for fifty years. Recently, a new variation of this operator has come under investigation.  Defining $D$ to be the operator of differentiation, several researchers have undertaken a study of the two operators $D\CO$ and $\CO D$ and we point to \cite{FatehiHammond:2020}, \cite{HibschweilerPortnoy:2005}, \cite{Ohno:2006}, and \cite{Ohno:2009}.  These studies have focused on the Hardy, Bergman, and Bloch spaces.  Our goal is to extend parts of this study to the Dirichlet space of the unit disk.

On many of the spaces mentioned above, the operator $D$ is unbounded and this behavior persists on the Dirichlet space. As a result of this, we adopt the convention from \cite{FatehiHammond:2020} and define the composition-differentiation operator $\CDO(f)=\CO D(f)=f'\circ\Ph.$ Throughout the sections that follow, we will examine when these operators are bounded, compact, and Hilbert-Schmidt on the Dirichlet space.  We then restrict our attention to two specific classes of symbols.  For linear-fractional maps with $\|\Ph\|_{\infty}<1$, we examine the adjoint of $\CDO$.  For maps induced by monomials, $\Ph(z)=az^M$ where $|a|<1$ and $M\in\N$, we determine the norm and spectrum of $\CDO$.  We also point to instances where our methods will extend to other spaces of interest.

In \cite{Ohno:2014}, the author studies composition operators from the Hardy space into the Dirichlet space and when such operators are bounded, compact, and Hilbert-Schmidt.  One byproduct of our results extends this work to composition operators from the Bergman space into the Dirichlet space.

\section{Preliminaries}

Let $\D=\{z\in\C:|z|<1\}$ be the open unit disk in the complex plane and $H(\D)$ be the set of all functions analytic on $\D$.  A function $f\in H(\D)$ is in the Dirichlet space $\DI$ provided $f$ has a finite Dirichlet integral, i.e. \[\int_{\D}|f'(z)|^2\,dA(z)<\infty,\] where $dA$ is normalized Lebesgue area measure. Occasionally we will use the symbol $\|f\|_{\DI_0}$ to denote the Dirichlet integral of $f$.  If $f(z)=\sum_{n=0}^{\infty} a_nz^n$ is the power series representation for $f$, then \[\begin{aligned}\|f\|_{\DI}^2&=|f(0)|^2+\int_\D|f'(z)|^2\,dA(z)\\
&=|a_0|^2+\sum_{n=1}^{\infty}n|a_n|^2.\\
\end{aligned}\] From this it is clear that we can view the Dirichlet space as a weighted Hardy space given by the weight sequence $\beta(0)=1$ and $\beta(n)=n^{1/2}$ for $n\geq 1$; see \cite[Chapter 2]{CowenMacCluer:1995} for more details. The generating function for this weight sequence is given by \[k(z)=\sum_{n=0}^{\infty}\frac{z^n}{\beta(n)^2}=1+\log\frac{1}{1-z}\] and hence the reproducing kernel function for the Dirichlet space (with norm and inner product induced by the given weight sequence) is \[K_w(z)=k(\overline{w}z)=1+\log\frac{1}{1-\overline{w}z}.\] With this, for each $f\in\DI$, we have \[f(w)=\langle f,K_w\rangle_{\DI},\] where \[\begin{aligned}\langle f,g\rangle_{\DI}&=f(0)\overline{g(0)}+\int_{\D}f'(z)\overline{g'(z)}\,dA(z)\\
&=a_0\overline{b_0}+\sum_{n=1}^{\infty}na_n\overline{b_n}.
\end{aligned}\]  Furthermore, the kernel function for evaluation of the first derivative is given by \[K_w^{(1)}(z)=\frac{d}{d\overline{w}}k(\overline{w}z)=\frac{z}{1-\overline{w}z}\] and \[f'(w)=\langle f,K_w^{(1)}\rangle_{\DI} \] for each $f\in \DI$.

Also recall the Bergman space of the disk $A^2$ is defined to be those functions analytic in the disk with $$\int_\D|f(z)|^2\,dA(z)<\infty.$$ Thus $f\in\DI$ if and only if $f'\in A^2$. The Bergman space is a functional Hilbert space and thus point evaluations are bounded (\cite[Chapter 2]{CowenMacCluer:1995}).

As we study composition-differentiation operators on $\DI$, we will need to employ a specific counting function. If $\Ph$ is self-map of $\D$, for $w\in\D$, we set $n_{\Ph}(w)$ to be the cardinality of the set $\{\Ph^{-1}(w)\}$; we note that $n_{\Ph}(w)=0$ for $w\not\in\Ph(\D)$.  By making a change of variables and using the fact that $\CDO(f)=f'\circ\Ph$ we see that \begin{equation}\label{eqn:countingfunctioninnorm}\begin{aligned}\|\CDO f\|_{\DI}^2&=|f'(\Ph(0))|^2+\int_\D|(f'\circ\Ph)'(z)|^2\,dA(z)\\
&=|f'(\Ph(0))|^2+\int_\D|f''(\Ph(z))|^2|\Ph'(z)|^2\,dA(z)\\\
&=|f'(\Ph(0))|^2+\int_\D|f''(w)|^2n_{\Ph(w)}\,dA(w).\\
\end{aligned}\end{equation}

To characterize bounded and compact composition-differentiation operators, we will employ Carleson measures. For $0\leq\theta<2\pi$ and $0<h<1$, define the standard Carleson set $S(\theta,h)$ by \[S(\theta,h)=\{z\in\D: |z-e^{i\theta}|<h\}.\] Our results center around these sets, but one can verify equivalent results using Carleson rectangles or pseudohyperbolic disks. 

For $a\in\D$, recall the involution automorphism $\Ph_a$ given by \[\Ph_a(z)=\frac{a-z}{1-\overline{a}z}\] satisfies \[|\Ph_a'(z)|=\frac{1-|a|^2}{|1-\overline{a}z|^2}.\] These maps are bijections of the disk onto itself and are self-inverses. More generally, we will consider linear fractional self-maps of the disk, i.e. maps of the form \begin{equation}\label{eqn:lfphi}\Ph(z)=\frac{az+b}{cz+d}\end{equation} that map the disk into itself. We point to \cite[Chapter 0]{Shapiro:1991} for more details. We will only be interested in non-constant maps of the disk into itself, and for maps of the form above, it is necessary that $|cz+d|$ be bounded away from zero throughout the disk. Additionally, we define the companion map \begin{equation}\label{eqn:lfsigma}\sigma_{\Ph}(z)=\sigma(z)=\frac{\overline{a}z-\overline{c}}{-\overline{b}z+\overline{d}}.\end{equation} It is simple to show that $\sigma$ is also a self-map of the disk whenever $\Ph$ is.  This companion map occurs frequently when considering the adjoint of a composition operator with linear fractional symbol acting on a variety of spaces and the interested reader is encouraged to consider \cite{Cowen:1988,Hurst:1997,GallardoRodriguez:2003,Heller:2012}

Finally, for two quantities $A(x)$ and $B(x)$, we say that $A(x)\cong B(x)$ to mean there exist positive constants $C_1$ and $C_2$ such that $C_1 A(x)\leq B(x)\leq C_2A(x)$ for all $x$ in some specified range.

\section{Bounded Composition-Differentiation Operators on $\DI$}

We begin with a result that is often referenced without proof but we include a proof for the sake of completeness. The result for $T=\CO$ appears as Exercise 1.1.1 in \cite{CowenMacCluer:1995}.

\begin{proposition}\label{prop:boundedinto}
Suppose $\mathcal{X}$ and $\mathcal{Y}$ are functional Banach spaces of analytic functions defined on $\D$.  Let $\Ph$ be a self-map of $\D$ and let $T$ be either $\CO$ or $\CDO$. Then $T$ is bounded from $\mathcal{X}$ to $\mathcal{Y}$ if and only if $T$ maps $\mathcal{X}$ into $\mathcal{Y}$.
\end{proposition}

\begin{proof}
The necessity is clear and we will verify sufficiency for $\CDO$. Suppose $\CDO$ maps $\mathcal{X}$ into $\mathcal{Y}$. To conclude that $\CDO$ is bounded, we will appeal to the Closed Graph Theorem.  Let $(f_n)$ be a sequence in $\mathcal{X}$ converging in norm to $f\in\mathcal{X}$.	Also suppose $(\CDO f_n)$ converges in norm to $g\in\mathcal{Y}$. For $x\in \D$, it is clear that \[(\CDO f_n)(x)=f_n'(\Ph(x))=K_{\Ph(x)}^{(1)}(f_n).\] Also, our assumptions on $(f_n)$ combined with the continuity of the evaluation functional guarantee that $\left(K_{\Ph(x)}^{(1)}(f_n)\right)$ converges to $K_{\Ph(x)}^{(1)}(f)$.  Thus, for each $x\in \D$, we have that $((\CDO f_n)(x))$ converges to \[K_{\Ph(x)}^{(1)}(f)=f'(\Ph(x))=(\CDO f)(x).\]

On the other hand, in a functional Banach space, norm convergence implies point-wise convergence and so, for $x\in\D$, we see that $((\CDO f_n)(x))$ converges to $g(x)$ and thus $(\CDO f)(x)=g(x)$ for every $x\in \D$ or $\CDO f=g$.  Therefore we conclude that $\CDO$ is bounded by the Closed Graph Theorem.
\end{proof}

\begin{theorem}
Let $\Ph$ be an analytic self-map of $\D$ and set $d\mu=n_{\Ph}dA$. The following are equivalent.

\begin{enumerate}
\item $D_\Ph:\DI\to \DI$ is bounded.
\item $\CO:A^2\to\DI$ is bounded.
\item There is a constant $C_1>0$ such that $\mu(S(\theta,h))\leq C_1h^4$ for all $0\leq\theta<2\pi$ and $0<h<1$.
\item There is a constant $C_2>0$ such that \[\int_{\D}|\Ph_a'(w)|^4\,d\mu(w)\leq C_2\] for all $a\in\D$.
\end{enumerate}
\end{theorem}

\begin{proof}
The equivalence $\text{(i)} \Longleftrightarrow \text{(ii)}$ follows from Proposition \ref{prop:boundedinto} combined with the fact that $f\in\DI$ if and only if $f'\in A^2$ and the observation that $D_\Ph(f)=f'\circ\Ph=\CO(f')$.

Next we verify the equivalence between  (ii) and (iii). Let $f\in A^2$.  Then, making a change of variables, we have \[\begin{aligned}\|\CO f\|_{\DI}^2&=|f(\Ph(0))|^2+\int_\D|(f\circ\Ph)'(z)|^2\,dA(z)\\
&=|f(\Ph(0))|^2+\int_\D|f'(\Ph(z))|^2|\Ph'(z)|^2\,dA(z)\\\
&=|f(\Ph(0))|^2+\int_\D|f'(w)|^2n_{\Ph(w)}\,dA(w)\\
&=|f(\Ph(0))|^2+\int_\D|f'(w)|^2\,d\mu(w).
\end{aligned}\] As point evaluations are bounded on $A^2$, it follows that $\CO:A^2\to \DI$ is bounded if and only if there is a constant $C>0$ such that \[\int_D|f'(w)|^2\,d\mu(w)\leq C\|f\|_{A^2}^2.\] The desired result then follows from \cite[Theorem 2.2]{Luecking:1985} (and the remarks following the theorem) with $n=1$, $q=2$, $p=2$, and $\alpha=0$.  The equivalence between (iii) and (iv) is found in \cite[Theorem 13]{ArazyFisherPeetre:1985}.
\end{proof}

For $\mu$ as defined in the theorem, we note that the integral in condition (iv) of the theorem can also be written as \begin{equation}\label{eqn:involutionintegral}\int_{\D}|\Ph_a'(w)|^4\,d\mu(w)=\int_{\D}\frac{|\Ph'(z)|^2}{(1-|\Ph(z)|^2)^4}\left(\frac{(1-|a|^2)(1-|\Ph(z)|^2)}{|1-\overline{a}\Ph(z)|^2}\right)^4\,dA(z)\end{equation} by reversing the change of variables.

\begin{example}\label{exam:bounded}
An obvious consequence of the previous theorem is that the identity map $\Ph(z)=z$ does not induced a bounded composition-differentiation operator on $\DI$, nor does any automorphism of the disk. This is also true for any linear fractional self-map of the disk with $\|\Ph\|_{\infty}=1$.   However, a linear fractional map with $\|\Ph\|_{\infty}<1$ will induce a bounded composition-differentiation operator.

The following example can be found in \cite{JovovicMacCluer:1997}. Let $\Gamma_1$ be a simply connected region in $\D$ touching $\partial\D$ only at 1 and such that the boundary curve in a neighborhood of 1 is a piece of the curve (in rectangular coordinates) $x^{1/2}+y^{1/2}=1$.  Then let $\Ph$ be a univalent mapping of $\D$ onto $\Gamma_1$.  In \cite{JovovicMacCluer:1997}, the authors show that such a $\Ph$ will induce a bounded and, in fact, compact composition operator on $\DI$. However, the corresponding operator $\CDO$ is unbounded on $\DI$.  This can be verified by showing that $\mu(S(\theta,h))/h^4$ is unbounded as $h\rightarrow 0$, for $\theta=0$. We have, \[\begin{aligned}\lim_{h\rightarrow 0}\frac{1}{h^4}\int_{S(0,h)}n_{\Ph}(w)\,dA(w)&=\lim_{h\rightarrow 0}\frac{1}{h^4}A(\Gamma_1\cap S(0,h))\\
&\cong\lim_{h\rightarrow 0}\int_{1-h}^1\left(1-x^{1/2}\right)^2\,dx\\
&=\lim_{h\rightarrow 0}\frac{2h-h^2/2+(4/3)(1-h)^{3/2}-4/3}{h^4}\\
&=\infty.\end{aligned}\]

We can adjust this example to produce a $\Ph$ such that $\CDO$ is bounded on $\DI$.  Here we take $\Gamma_2$ to be defined similarly to $\Gamma_1$ but we modify so that the boundary curve in a neighborhood of 1 is a piece of the curve $x^{1/3}+y^{1/3}=1.$ Then we take $\Ph$ to be a univalent map of $\D$ onto $\Gamma_2$.  To show that $\CDO$ is bounded, it is sufficient to check that $\mu(S(0,h)/h^4$ is bounded for all $0<h<1$ since $\Gamma_2$ is contained in a non-tangential approach region at 1; see, for example, \cite[Lemma 6.2]{CowenMacCluer:1995}. We have \[\begin{aligned}\frac{1}{h^4}\int_{S(0,h)}n_{\Ph}(w)\,dA(w)&=\frac{1}{h^4}A(\Gamma_2\cap S(0,h))\\
&\cong\int_{1-h}^1\left(1-x^{1/3}\right)^3\,dx\\
&=\frac{h^2/2+(9/4)(1-h)^{4/3}-(9/5)(1-h)^{5/3}-9/20}{h^4}.\end{aligned}\] This last function is increasing on $(0,1)$ and so it is bounded above by the value at $h=1$ which is 1/20.
\end{example}

\section{Compact and Hilbert-Schmidt Composition-Differentiation Operators on $\DI$}

To characterize the compact composition-differentiation operators, we first have a familiar reformulation of compactness.

\begin{proposition}[{\cite[Lemma 3.7]{Tjani:2003}}]\label{prop:compactreformulation}
Let $\Ph$ be a self-map of $\D$. Then $\CO:A^2\to \DI$ is compact if and only if whenever $(f_n)$ is a bounded sequence in $A^2$,with $(f_n)\to 0$ uniformly on compact subsets of $\D$, it follows that $(\|\CO f_n\|_{\DI})\to 0$. Similarly, $\CDO:\DI\to\DI$ is compact if and only if whenever $(f_n)$ is a bounded sequence in $\DI$ with $(f_n)\to 0$ uniformly on compact subsets of $\D$, it follows that $(\|\CDO f_n\|_{\DI})\to 0$.
\end{proposition}

\noindent As with boundedness, compactness of $\CDO:\DI\to \DI$ is connected to the compactness of $\CO:A^2\to\DI$.

\begin{lemma}
Let $\Ph$ be a self-map of $\D$.  Then $\CDO:\DI\to \DI$ is compact if and only if $\CO:A^2\to\DI$ is compact.
\end{lemma}

\begin{proof}
Suppose $D_\Ph:\DI\to \DI$ is compact. To show $\CO:A^2\to\DI$ is compact, we will appeal to Proposition \ref{prop:compactreformulation}. To that end, let $(f_n)$ be a bounded sequence in $A^2$ such that $(f_n)\to 0$ uniformly on compact subsets of $\D$.  For $n\in \N$, let $g_n\in \DI$ with $g_n'=f_n$, i.e. take \[g_n(z)=\int_0^z f_n(w)\,dw.\]  Then $(g_n)$ is a bounded sequence in $\DI$ since $\|g_n\|_{\DI}^2=|g_n(0)|^2+\|g_n'\|_{A^2}^2=\|f_n\|_{A^2}^2$. The hypothesis on $(f_n)$ guarantees that $(g_n)\to 0$ uniformly on compact subsets of $\D$.  Thus $(\|\CDO g_n\|_{\DI})\to 0$ by Proposition \ref{prop:compactreformulation}.  Also, we have \[\|\CO f_n\|_{\DI}=\|\CO g_n'\|_{\DI}=\|g_n'\circ\Ph\|_{\DI}=\|\CDO g_n\|_{\DI}\] and thus $(\|\CO f_n\|_{\DI})\to 0$. Thus $\CO:A^2\to \DI$ is compact by Proposition \ref{prop:compactreformulation}.

Conversely, suppose $\CO:A^2\to \DI$ is compact.  To show $\CDO:\DI\to \DI$ is compact, let $(g_n)$ be a bounded sequence in $\DI$ such that $(g_n)\to 0$ uniformly on compact subsets of $\D$. Then $(g_n')$ is a bounded sequence in $A^2$ since $\|g_n'\|_{A^2}=\|g_n\|_{\DI_0}\leq \|g_n\|_{\DI}$. Moreover $(g_n')\to 0$ uniformly on compact subsets of $\D$; see \cite[Theorem V.1.6]{Lang:1993}.  By our hypothesis and Proposition \ref{prop:compactreformulation} we have $(\|\CO g_n'\|_{\DI})\to 0$.  But \[\|\CDO g_n\|_{\DI}=\|g_n'\circ \Ph\|_{\DI}=\|\CO g_n'\|_{\DI}\] and hence $(\|\CDO g_n\|_{\DI})\to 0$.  Thus $\CDO:\DI\to \DI$ is compact.
\end{proof}

The lemma leads to the following characterization of compactness for $\CDO:\DI\to\DI$.

\begin{theorem}\label{thm:compactness}
Let $\Ph$ be a self-map of $\D$ so that $\CDO:\DI\to \DI$ is bounded and set $d\mu=n_{\Ph}dA$. The following are equivalent.

\begin{enumerate}
\item $D_\Ph:\DI\to \DI$ is compact.
\item $\CO:A^2\to\DI$ is compact.
\item $\displaystyle \lim_{h\to 0} \sup_{\theta\in[0,2\pi)}\frac{\mu(S(\theta,h))}{h^4}=0$ .
\item $\displaystyle \lim_{|a|\to 1}\int_{\D}|\Ph_a'(w)|^4\,d\mu(w)=0$.
\end{enumerate}
\end{theorem}

\begin{proof}
By the lemma we see that (i) and (ii) are equivalent and we will show $\text{(ii)}\Longrightarrow\text{(iv)}\Longrightarrow\text{(iii)}\Longrightarrow\text{(ii)}$.  Suppose $\CO:A^2\to \DI$ is compact. For $a\in\D$ let $k_a(z)=(1-|a|^2)/(1-\overline{a}z)^2$.  Then $k_a\in A^2$ with $\|k_a\|_{A^2}=1$ and $k_a$ converges weakly to 0 as $|a|\to 1$ \cite[Theorem 2.17]{CowenMacCluer:1995}. Hence $(\|\CO k_a\|_{\DI})\to 0$ as $|a|\to 1$ since $\CO:A^2\to \DI$ is compact. Then we have \[\begin{aligned}\|\CO k_a\|_{\DI}^2\geq \|(k_a\circ \Ph)'\|_{A^2}^2&=\int_{\D}\frac{4|a|^2(1-|a|^2)^2|\Ph'(z)|^2}{|1-\overline{a}\Ph(z)|^6}\,dA(z)\\
&=4|a|^2\int_{\D}\frac{|\Ph'(z)|^2(1-|a|^2)^4(1-|\Ph(z)|^2)^4}{(1-|\Ph(z)|^2)^4|1-\overline{a}\Ph(z)|^8}\frac{|1-\overline{a}\Ph(z)|^2}{(1-|a|^2)^2}\,dA(z)\\
&\geq4|a|^2\int_{\D}\frac{|\Ph'(z)|^2(1-|a|^2)^4(1-|\Ph(z)|^2)^4}{(1-|\Ph(z)|^2)^4|1-\overline{a}\Ph(z)|^8}\frac{(1-|a|)^2}{(1-|a|^2)^2}\,dA(z)\\
&\geq |a|^2\int_{\D}\frac{|\Ph'(z)|^2(1-|a|^2)^4(1-|\Ph(z)|^2)^4}{(1-|\Ph(z)|^2)^4|1-\overline{a}\Ph(z)|^8}\,dA(z)\\
&=|a|^2\int_{\D}\frac{|\Ph'(z)|^2}{(1-|\Ph(z)|^2)^4}\left(\frac{(1-|a|^2)(1-|\Ph(z)|^2)}{|1-\overline{a}\Ph(z)|^2}\right)^4\,dA(z)\\
&=|a|^2\int_{\D}|\Ph_a'(w)|^4\,d\mu(w),
\end{aligned}\] where the last equality follows from Eqn.(\ref{eqn:involutionintegral}).  It follows that (iv) holds as $|a|\to 1$.

That (iv) implies (iii) is shown in \cite[Proposition 3.4]{Tjani:2003}. Lastly, we will show that (iii) implies (ii). To show $\CO:A^2\to \DI$ is compact, we will again appeal to Proposition \ref{prop:compactreformulation}.  Let $(f_n)$ be a bounded sequence in $A^2$ with $(f_n)\to 0$ uniformly on compact subsets of $\D$. Then we are left to show that $(\|\CO f_n\|_{\DI})\to 0$.  By our assumptions on $(f_n)$ we know $(|f_n(\Ph(0))|^2)\to 0$ and thus it is sufficient to show that $(\|\CO f_n\|_{\DI_0})\to 0$. Let $w\in\D$, $r=(1-|w|)/2$, and $D(w,r)=\{z\in\D: |z-w|<r\}$.  As $|f_n'|^2$ is subharmonic, for $w\in\D$, we have \[|f_n'(w)|^2\leq \frac{1}{r^2}\int_{D(w,r)}|f_n'(z)|^2\,dA(z)=\frac{4}{(1-|w|)^2}\int_{D(w,r)}|f_n'(z)|^2\,dA(z).\]  Then, from our previous estimate and Fubini's theorem, \[\begin{aligned}\|\CO f_n\|^2_{\DI_0}&=\int_\D|f_n'(\Ph(z))|^2|\Ph'(z)|^2\,dA(z)=\int_\D|f_n'(w)|^2\,d\mu(w)\\
&\leq\int_\D\frac{4}{(1-|w|)^2}\left(\int_{D(w,r)}|f_n'(z)|^2\,dA(z)\right)\,d\mu(w)\\
&=4\int_\D\frac{1}{(1-|w|)^2}\left(\int_\D\bigchi_{D(w,r)}(z)|f_n'(z)|^2\,dA(z)\right)\,d\mu(w)\\
&=4\int_\D|f_n'(z)|^2\left(\int_\D\frac{\bigchi_{D(w,r)}(z)}{(1-|w|)^2}\,d\mu(w)\right)\,dA(z).\end{aligned}\] If $z,w\in \D$ with $|w-z|<(1-|w|)/2$, then $(1-|w|)/2<1-|z|$ and, for $z=|z|e^{i\theta}$, we have \[|w-e^{i\theta}|\leq |w-z|+|z-e^{i\theta}|\leq \frac{1-|w|}{2}+(1-|z|)\leq 2(1-|z|).\] Thus, for a fixed $z\in\D$, if $w$ satisfies $|w-z|<(1-|w|)/2$, we have $w\in S(\theta,2(1-|z|))$.  Now define functions $F$ and $G$ on $\D\times \D$ by $F(z,w)=\bigchi_{D(w,r)}(z)$, where $r=(1-|w|)/2$, and $G(z,w)=\bigchi_{S(\theta,2(1-|z|))}(w)$, where $z=|z|e^{i\theta}$.  If $z$ is fixed and $w$ satisfies $z\in D(w,r)$, then $F(z,w)=1=G(z,w)$.  Otherwise $F(z,w)=0\leq G(z,w)$.  Thus $F(z,w)\leq G(z,w)$ on $\D\times \D$.  Additionally, if $|w-z|<(1-|w|)/2$, then $1/(1-|w|)\leq (3/2)1/(1-|z|)$. With these two estimates, we have \[\begin{aligned}\|\CO f\|^2_{\DI_0}&\leq 9\int_\D\frac{|f_n'(z)|^2}{(1-|z|)^2}\left(\int_\D\bigchi_{S(\theta,2(1-|z|))}(w)\,d\mu(w)\right)\,dA(z)\\
&=9\int_\D\frac{|f_n'(z)|^2}{(1-|z|)^2}\left(\int_{S(\theta,2(1-|z|))}\,d\mu(w)\right)\,dA(z).\\
\end{aligned}\]

Now, with (iii) as our assumption, let $\varepsilon>0$ and choose $\delta>0$ such that $\mu(S(\theta,h))<\varepsilon h^4$ for $0<h<\delta$ and $\theta\in[0,2\pi)$. Then we can split the last integral and we have \[\|\CO f\|_{\DI_0}\leq I_1 + I_2,\] where \[I_1=9\int_{|z|>1-\delta/2}\frac{|f_n'(z)|^2}{(1-|z|)^2}\left(\int_{S(\theta,2(1-|z|))}\,d\mu(w)\right)\,dA(z)\]
and \[I_2=9\int_{|z|\leq1-\delta/2}\frac{|f_n'(z)|^2}{(1-|z|)^2}\left(\int_{S(\theta,2(1-|z|))}\,d\mu(w)\right)\,dA(z).\] Utilizing our choice of $\delta$, we see that there is  a positive constant $C$ such that \[\begin{aligned}I_1&\leq\varepsilon2^4 9\int_{|z|>1-\delta/2}\frac{|f_n'(z)|^2}{(1-|z|)^2}(1-|z|)^4\,dA(z)\\
&\leq \varepsilon 2^49\int_\D|f_n'(z)|^2(1-|z|^2)^2\,dA(z)\\
&\leq C\varepsilon,
\end{aligned}\] where the last integral represents an equivalent (semi-)norm on $A^2$ and is therefore bounded as $n$ varies by our conditions on $(f_n)$.

Also, \[\begin{aligned}I_2&\leq \frac{36}{\delta^2}\int_{|z|\leq1-\delta/2}|f_n'(z)|^2\left(\int_\D\,d\mu(w)\right)\,dA(z)\\
&=\frac{36}{\delta^2}\|\Ph\|_{\DI_0}^2\int_{|z|\leq1-\delta/2}|f_n'(z)|^2\,dA(z).\\
\end{aligned}\] By our assumptions on $(f_n)$, we know that $(f_n')\to 0$ uniformly on compact subsets of $\D$ (see \cite[Theorem V.1.6]{Lang:1993}) and hence we can choose $n\in\N$ large enough so that this last integral is bounded above by a constant multiple of $\varepsilon$.  Thus we conclude that $(\|\CO f_n\|_{\DI_0})\to 0$ and hence $\CO:A^2\to \DI$ is compact as desired.
\end{proof}

The following corollary follows from either Proposition \ref{prop:compactreformulation} or Theorem \ref{thm:compactness}

\begin{corollary}\label{cor:compact}
Let $\Ph$ be a self-map of $\D$ so that $\CDO:\DI\to \DI$ is bounded. If $\|\Ph\|_{\infty}<1$, then $\CDO:\DI\to \DI$ is compact.
\end{corollary}

\begin{example}
To generate a compact $\CDO$ for a self-map with boundary contact, we can modify the last self-map from Example \ref{exam:bounded}.  We take $\Gamma_3$ to be similar to $\Gamma_2$ from the aforementioned example but we use the curve $x^{1/4}+y^{1/4}=1$ as the boundary curve in a neighborhood of 1.  We then let $\Ph$ be a univalent map of $\D$ onto $\Gamma_3$.  Again, since $\Gamma_3$ is contained in a non-tangential approach region at 1, we only need to verify condition (iii) from Theorem \ref{thm:compactness} at $\theta=0$, i.e we need $\lim_{h\rightarrow 0}\mu(S(0,h))/h^4=0$. The computation is similar to those in the previous example and we omit the details.
\end{example}

Characterizing which $\CDO$ induce Hilbert-Schmidt operators on $\DI$ is a consequence of \cite[Theorem 3.23]{CowenMacCluer:1995}. As is the case for composition operators acting on $\DI$, this result has connections to the hyperbolic derivative of the inducing map.

\begin{theorem}
Let $\Ph$ be a self-map of $\D$.  Then $\CDO:\DI\to \DI$ is Hilbert-Schmidt if and only if \[\int_{\D}\frac{|\Ph'(z)|^2}{(1-|\Ph(z)|^2)^4}\,dA(z)<\infty.\]
\end{theorem}

\begin{proof}
We know that the sequence of functions $(e_n)$ defined by $e_0(z)=1$ and $e_n(z)=z^n/\sqrt{n}$, for $n\geq 1$, is an orthonormal basis for $\DI$. It follows that $\CDO$ is Hilbert-Schmidt on $\DI$ if and only if \[\sum_{n=0}^{\infty}\|\CDO e_n\|_{\DI}^2<\infty.\] Considering \myeqref{eqn:countingfunctioninnorm}, we see further that $\CDO$ is Hilbert-Schmidt if and only if  \[\sum_{n=2}^{\infty}\|\CDO e_n\|_{\DI_0}^2=\sum_{n=0}^{\infty}\int_\D(n+2)(n+1)^2|\Ph(z)|^{2n}|\Ph'(z)|^2\,dA(z)<\infty.\] Differentiating the geometric series three times,  we have \[\sum_{n=0}^{\infty}(n+3)(n+2)(n+1)z^n=\frac{3!}{(1-z)^4},\] for $z\in\D$. Also, for $n\geq 0$, we have $1/3\leq (n+1)/(n+3)\leq 1$ and it follows that $\sum_{n=2}^{\infty}\|\CDO e_n\|_{\DI_0}^2$ is bounded above and below by multiples of \[\sum_{n=0}^{\infty}\int_\D(n+3)(n+2)(n+1)|\Ph(z)|^{2n}|\Ph'(z)|^2\,dA(z)=3!\int_\D\frac{|\Ph'(z)|^2}{(1-|\Ph(z)|^2)^4}\,dA(z).\] The conclusion now follows.
\end{proof}

\section{Adjoints}

In this section we will consider the adjoint of a composition-differentiation operator induced by a linear fractional map $\Ph$ with $\|\Ph\|_{\infty}<1$; we assume $\Ph$ is non-constant and ask the reader to recall Eqs. (\ref{eqn:lfphi}) and (\ref{eqn:lfsigma}).  Note any such map will induce a bounded (and, in fact, compact)  composition-differentiation operator on $\DI$. Moreover, the conditions on $\Ph$, and hence on $\sigma$, guarantee that both $|cz+d|$ and $|-\overline{b}z+\overline{d}|$ are bounded away from 0 in the disk.

For $\psi$ analytic in $\D$, we define the multiplication operator $T_\psi$ by \[T_\psi(f)=\psi\cdot f.\] The following lemma gives a sufficient condition for $T_\psi$ to be bounded on the Dirichlet space.

\begin{lemma}\label{lemma:boundedmultiplicationoperator}
Let $\psi$ be analytic in $\D$ with both $\psi$ and $\psi'$ in $H^{\infty}(\D)$. Then $T_{\psi}$ is a bounded multiplication operator on $\DI$.
\end{lemma}

For a bounded multiplication operator acting on a reproducing kernel Hilbert space, the adjoint of the multiplication operator has predictable behavior when acting on a kernel function; in particular, $T_\psi^*(K_w)=\overline{\psi(w)}K_w$. There are similar results for the adjoint of a composition or a composition-differentiation operator: $\CO^*(K_w)=K_{\Ph(w)}$ and $\CDO^*(K_w)=K_{\Ph(w)}^{(1)}$. We will utilize these facts to prove the following theorem, which is an analogue of \cite[Theorem 1]{FatehiHammond:2020}

\begin{theorem}
Let $\Ph$ be a linear fractional self-map of $\D$ with $\|\Ph\|_{\infty}<1$. Then $D_{\Ph}^*T_{K_{\SI(0)}^{(1)}}^*=T_{K_{\Ph(0)}^{(1)}}D_{\SI}$.
\end{theorem}

\begin{proof}
For $\Ph$ and $\SI$ as given above, $\Ph(0)=b/d$ and $\SI(0)=-\overline{c}/\overline{d}$, from which we have \[ K_{\Ph(0)}^{(1)}(z)=\frac{\overline{d}z}{-\overline{b}z+\overline{d}}\] and \[K_{\SI(0)}^{(1)}(z)=\frac{dz}{cz+d}.\]  Note that both of these functions are bounded in $\D$, as are their respective derivatives, by the remarks preceding Lemma \ref{lemma:boundedmultiplicationoperator} and hence the induced multiplication operators are bounded on $\DI$ by the aforementioned lemma. It follows immediately that \[D_{\Ph}^*T_{K_{\SI(0)}^{(1)}}^*(K_w)=\overline{K_{\SI(0)}^{(1)}(w)}D_{\Ph}^*(K_w)=\frac{\overline{dw}}{\overline{cw}+\overline{d}}K_{\Ph(w)}^{(1)}.\]  Next, observe that \[D_{\SI}(K_w)(z)=C_{\SI}\left(\frac{\overline{w}}{1-\overline{w}z}\right)=\frac{\overline{w}}{1-\overline{w}\SI(z)}\] and so \[\begin{aligned}T_{K_{\Ph(0)}^{(1)}}D_{\SI}(K_w)(z)&=\left(\frac{\overline{d}z}{-\overline{b}z+\overline{d}}\right)\left(\frac{\overline{w}}{1-\overline{w}\SI(z)}\right)\\
&=\frac{\overline{dw}z}{\overline{cw}+\overline{d}-(\overline{aw}+\overline{b})z}\\
&=\left(\frac{\overline{dw}}{\overline{cw}+\overline{d}}\right)\left(\frac{z}{1-\overline{\Ph(w)}z}\right)\\
&=\frac{\overline{dw}}{\overline{cw}+\overline{d}}K_{\Ph(w)}^{(1)}(z).
\end{aligned}\] The kernel functions span a dense set in $\DI$ and hence we conclude that $D_{\Ph}^*T_{K_{\SI(0)}^{(1)}}^*=T_{K_{\Ph(0)}^{(1)}}D_{\SI}$ on $\DI$.
\end{proof}

For composition operators induced by linear fractional symbol acting on the Dirichlet space, the adjoint has the form $\CO^*=C_{\sigma}+K$, where $K$ is a specific rank 2 operator (see \cite[Theorem 3.3]{GallardoRodriguez:2003}), whereas on the Hardy space the adjoint for the composition and composition-differentiation operators seem to have more similarity.  Specifically, we have $\CO^*T_{K_\sigma(0)}^*=T_{K_\Ph(0)}C_{\sigma}$ and $D_{\Ph}^*T_{K_{\SI(0)}^{(1)}}^*=T_{K_{\Ph(0)}^{(1)}}D_{\SI}$, respectively, where the kernels appearing here are those for the Hardy space (see \cite[Theorem 2]{Cowen:1988} and \cite[Theorem 1]{FatehiHammond:2020}).

\section{Special Symbols}\label{Section:SpecialSymbols}

In this section we investigate composition-differentiation operators induced by a monomial symbol.  Specifically, we explore the norm and spectrum of such an operator.

\subsection{Norm}\label{Subsection:Norm}

We first consider the norm of the operator $D_\varphi$ acting on $\DI$ when $\varphi(z) = az^M$ for $0 < |a| < 1$ and $M\in\N$. Note that any such $\varphi$ will induce a compact operator on $\DI$ by Corollary \ref{cor:compact}. In the case of the Hardy space, the authors in \cite{FatehiHammond:2020} obtain a similar result for $\CDO$ acting on the Hardy space $H^2$. In that setting, they utilize certain facts about composition and multiplication operators on $H^2$ that do not hold on the Dirichlet space (or on the standard weighted Bergman or Dirichlet spaces). Specifically, the authors there rely on the fact that $T_z^*T_z=I$ for $T_z$ acting on the Hardy space. Instead, we will exploit the fact that the operator $\CDO$, for $\Ph$ as above, preserves the orthogonality of the monomials.

For $0<|a|<1$ and $M\in\N$, we define constants 
\[\nu = \left\lfloor\frac{2}{1-|a|^2}\right\rfloor\]
and 
\[\mathcal{N}_M = \max\left\{1, \sqrt{M}\sqrt{\nu(\nu-1)}|a|^{\nu-1}\right\}.\]

\begin{theorem} If $\varphi(z) = az^M$ for $0 < |a| < 1$ and $M$ in $\N$, then $\|D_\varphi\| = \mathcal{N}_M$.
\end{theorem}

\begin{proof}
For $\varphi$ as given, we will show $\|D_\varphi\| \geq \mathcal{N}_M$ and $\|D_\varphi\| \leq \mathcal{N}_M$. To obtain the lower bound we first consider the orthonormal basis functions $(e_n)_{n\geq 0}$ defined by $e_0(z) = 1$ and $e_n(z) = z^n/\sqrt{n}$, $n\geq 1$. For $z\in \D$, observe $e_0'(\varphi(z))=0$,
\[e_1'(\varphi(z)) = 1 = e_0(z),\]
and, for $n \geq 2$,
\[e_n'(\varphi(z)) = \frac{n\left(az^M\right)^{n-1}}{\sqrt{n}} = \sqrt{n}a^{n-1}z^{M(n-1)} =  \sqrt{M}\sqrt{n(n-1)}a^{n-1}e_{M(n-1)}(z).\] Thus \[\|D_\varphi\| \geq \max\left\{1,\sup_{n \geq 2}\sqrt{M}\sqrt{n(n-1)}|a|^{n-1}\right\}.\] The lower bound will follow if we can verify that \[\sup_{n \geq 2}\sqrt{M}\sqrt{n(n-1)}|a|^{n-1} = \sqrt{M}\sqrt{\nu(\nu-1)}|a|^{\nu-1}.\] Note that the function $g(x) = \sqrt{M}\sqrt{x(x-1)}|a|^{x-1}$ has exactly one critical point in $(1,\infty)$. This point is a local maximum and hence also the absolute maximum of $g$ on $(1,\infty)$.  Thus the supremum we wish to compute is a maximum and will occur at the greatest integer $n \geq 2$ such that 
\[\sqrt{(n-1)(n-2)}|a|^{n-2} \leq \sqrt{n(n-1)}|a|^{n-1};\] equivalently,
\[n \leq \frac{2}{1-|a|^2}.\]
Thus, \begin{equation}\label{eqn:supremum}\sup_{n \geq 2}\sqrt{M}\sqrt{n(n-1)}|a|^{n-1} = \sqrt{M}\sqrt{\nu(\nu-1)}|a|^{\nu-1},\end{equation} which implies $\|D_\varphi\| \geq \mathcal{N}_M$. 

For the upper estimate, let $f(z) = \sum_{n=0}^{\infty} b_nz^n= b_0e_0(z)+\sum_{n=1}^{\infty}b_n\sqrt{n}e_n(z)$ be in $\DI$. Then 
\[\begin{aligned}(D_\varphi f)(z)&= \sum_{n=1}^{\infty} b_n\sqrt{n}e_n'(\varphi(z))\\
&=b_1+\sum_{n=2}^{\infty}b_n\sqrt{n}\left(\sqrt{M}\sqrt{n(n-1)}a^{n-1}e_{M(n-1)}(z)\right).
\end{aligned}\] With this and \myeqref{eqn:supremum}, we have  
\[\begin{aligned}
\|D_\varphi f\|_{\DI}^2 &= |b_1|^2+\sum_{n=2}^{\infty}|b_n|^2n\left(\sqrt{M}\sqrt{n(n-1)}|a|^{n-1}\right)^2\\
&\leq |b_1|^2+\sum_{n=2}^{\infty}|b_n|^2n\left(\sqrt{M}\sqrt{\nu(\nu-1)}|a|^{\nu-1}\right)^2\\
&\leq \mathcal{N}_M^2\left(|b_1|^2+\sum_{n=2}^{\infty}|b_n|^2n\right) \\
&\leq \mathcal{N}_M^2 \|f\|_{\DI}^2\\
\end{aligned}\] and thus $\|D_\varphi\|\leq \mathcal{N}_M$ as desired.
\end{proof}

Further investigation of the quantities in the theorem reveal how the norm changes with respect to $|a|$.  When $M=1$, $\|D_\varphi\|=1$ when $0<|a|\leq6^{-1/4}$ and $\|D_\varphi\|>1$ for $6^{-1/4}<|a|<1$. For $M\geq 2$, $\|D_\varphi\|=1$ when $0<|a|\leq1/\sqrt{2M}$ and $\|D_\varphi\|>1$ for $1/\sqrt{2M}<|a|<1$. In both cases, the norm tends to infinity as $|a|\to1$; see Figure \ref{Figure:NormDphi} below.

As alluded to before the statement of the theorem, our proof here does not rely on operator-theoretic properties of multiplication operators that are specific to a certain space, but rather on the fact that, for these symbols, the composition-differentiation operator preserves the orthogonality of the monomials.  This method applies to the Hardy space as well as a variety of other spaces.

\begin{figure}
\begin{center}
\resizebox {.9\textwidth} {\height} {
\begin{tikzpicture}[baseline]
    \begin{axis}[
	    axis x line=middle,
		axis y line=middle,
		axis on top,
		width=.75\textwidth,
		ylabel={\footnotesize $\|D_\varphi\|$},
		xlabel={\footnotesize $|a|$},
		ytick={1,3},
		yticklabels={1,3},
		xtick={0,0.408248,.5,0.638943,1},
		xticklabels={0,$\frac{1}{\sqrt{6}}$,$\frac{1}{\sqrt{4}}$,$\frac{1}{\sqrt[4]{6}}$,$1$},
		xtick align=outside,
		xmin=0, xmax=1.05, ymin=0, ymax=4.25,
		]
		\addplot[black,ultra thick,domain=0:0.95] {max(1,sqrt(floor(2/(1-x^2))*(floor(2/(1-x^2))-1))*x^(floor(2/(1-x^2))-1))};
		\addplot[black,ultra thick,domain=0:0.9] {max(1,sqrt(2)*sqrt(floor(2/(1-x^2))*(floor(2/(1-x^2))-1))*x^(floor(2/(1-x^2))-1))};
		\addplot[black,ultra thick,domain=0:0.9] {max(1,sqrt(3)*sqrt(floor(2/(1-x^2))*(floor(2/(1-x^2))-1))*x^(floor(2/(1-x^2))-1))};
		\node at (axis cs:0.638943,1) [anchor=north] {\tiny $M$=$1$};
		\node at (axis cs:.5,1) [anchor=north] {\tiny $M$=$2$};
		\node at (axis cs:0.408248,1) [anchor=north] {\tiny $M$=$3$};
	    %\addplot[black, style=dotted] coordinates {(0.638943,0) (0.638943,1)};
		%\addplot[black, style=dotted] coordinates {(0.5,0) (0.5,1)};
		%\addplot[black, style=dotted] coordinates {(0.408248,0) (0.408248,1)};
	\end{axis}
\end{tikzpicture}
}
\caption{Norm of $D_\varphi$ for $\varphi(z) = az^M$, as a function of $|a|$, with $M=1,2,3$.}\label{Figure:NormDphi}
\end{center}
\end{figure}
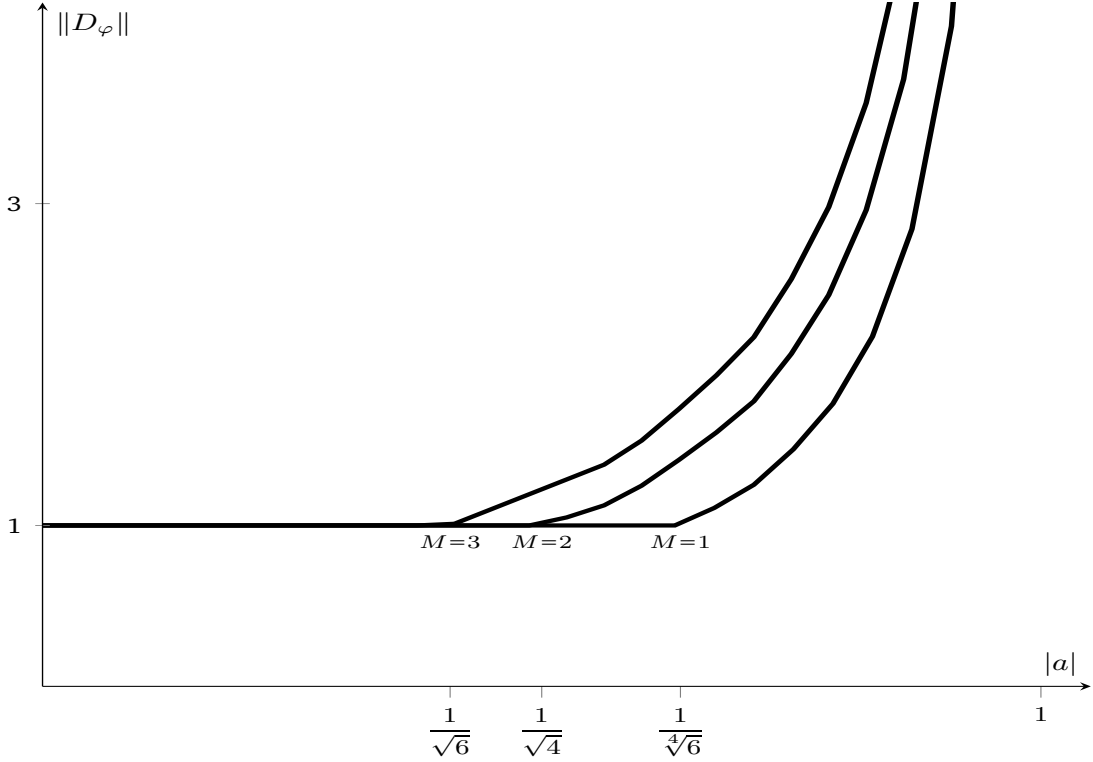

\subsection{Spectrum}\label{Subsection:Spectrum}

Since a self-map of the disk $\Ph$ with $\|\varphi\|_\infty < 1$ induces a compact operator $D_\varphi$ on $\DI$ by Corollary \ref{cor:compact}, as it does on the Hardy space $H^2$, the proof of the following theorem follows exactly as presented in \cite[Proposition 3]{FatehiHammond:2020}. 

\begin{theorem}\label{Theorem:SpectrumLinear}
If $\varphi(z) = az+b$, for $0 < |a| < 1-|b|$, then the spectrum of $D_\varphi$ is $\{0\}$.
\end{theorem}

We now turn our attention to the spectrum of $\CDO$ with symbols of the form $\varphi(z) = az^M$, for $0 < |a| < 1$ and $M\in\N$.  The previous theorem deals with the case $M=1$ and we are left to investigate the spectrum when $M \geq 2$.  By the Spectral Theorem for Compact Operators (see \cite{Conway:1990}), the spectrum is countable, contains 0, and any non-zero element is an eigenvalue. If we assume the existence of a non-zero eigenvalue $\lambda$ of $D_\varphi$, then any associated eigenfunction $f$ satisfies \[\frac{d^n}{dz^n}\left[\lambda f(z)\right] = \frac{d^n}{dz^n}\left[f'(\varphi(z))\right]\] for all $z$ in $\D$ and $n \geq 0$.

We will compute high-order derivatives of the function $f'\circ\varphi$ utilizing Fa\`a di Bruno's Formula (see \cite[Theorem C of Section 3.4]{Comtet:1974})
\begin{equation}\label{Equation:FaadiBruno}\frac{d^n}{dz^n} f'(\varphi(z)) = \sum_{k=1}^n f^{(k+1)}(\varphi(z))B_{n,k}\left(\varphi'(z),\dots,\varphi^{(n-k+1)}(z)\right),\end{equation} where $B_{n,k}(x_1,\dots,x_{n-k+1})$ is the Bell polynomial (see \cite[Theorem A of Section 3.3]{Comtet:1974}) defined as 
\begin{equation}\label{Equation:BellPolynomial}B_{n,k}(x_1,\dots,x_{n-k+1}) = \sum\frac{n!}{c_1!c_2!\cdots c_{n-k+1}!}\left(\frac{x_1}{1!}\right)^{c_1}\cdots\left(\frac{x_{n-k+1}}{(n-k+1)!}\right)^{c_{n-k+1}}
\end{equation} The summation in \myeqref{Equation:BellPolynomial} is taken over all sequences $c_1, \dots, c_{n-k+1}$ of non-negative integers that satisfy
\begin{equation}\label{Equation:SequenceRelations}
\begin{cases}
c_1 + c_2 + \cdots + c_{n-k+1} = k\\
1c_1 + 2c_2 + \cdots + (n-k+1)c_{n-k+1} = n.
\end{cases}
\end{equation}
Of particular interest in this section is the relation
\[\frac{d^n}{dz^n}\left[\lambda f(z)\right]{\bigg\rvert}_{z = 0} = \frac{d^n}{dz^n}\left[f'(\varphi(z))\right]{\bigg\rvert}_{z = 0}.\]  Combining this with \myeqref{Equation:FaadiBruno} and the fact that $\varphi(0)=0$ yields that an eigenfunction $f$ associated with a non-zero eigenvalue $\lambda$ must satisfy
\begin{equation}\label{Equation:a0Relation}
f(0) = \frac{1}{\lambda}f'(0)
\end{equation}
and
\begin{equation}\label{Equation:1stanRelation}
f^{(n)}(0) = \frac{1}{\lambda}\sum_{k=1}^n f^{(k+1)}(0)B_{n,k}\left(\varphi'(0),\dots,\varphi^{(n-k+1)}(0)\right)
\end{equation} for all $n$ in $\N$.

Since all but one derivative of $\varphi$ will be zero at $z=0$, many of the Bell Polynomials present in \eqref{Equation:anRelation} will have many, if not all, the $x_j$ terms being 0.  Those Bell polynomials used in this section are collected here for convenience. 

\begin{lemma}\label{Lemma:BellPolynomialValues}
If $j,k,n$ in $\N$ with $n > 1$, $k \leq n$, and $1 < j < n-k+1$, then 
\begin{enumerate}
\item $B_{n,k}(0,\dots,0) = 0$,
\item $B_{n,1}(0,\dots,0,x_n) = x_n$,
\item $B_{n,1}(0,\dots,0,x_j,0,\dots,0) = 0$,
\item $B_{n,n-1}(0,x_2) = 0$ for $n \neq 2$.
\end{enumerate}
\end{lemma}

\begin{remark}\label{Remark:BellPolynomialReduction}
Since the symbol $\varphi(z) = az^M$ with $M \geq 2$, it is the case that $\varphi^{(\ell)}(0) = 0$ for all $\ell \neq M$.  In particular, $\varphi'(0) = 0$.  From Lemma \ref{Lemma:BellPolynomialValues} we obtain  \[B_{n,n}\left(\varphi'(0),\dots,\varphi^{n-k+1}(0)\right) = B_{n,n}(\varphi'(0)) = B_{n,n}(0) = 0.\] Thus, \myeqref{Equation:1stanRelation} reduces to
\[f'(0) = \frac{1}{\lambda}f''(0)B_{1,1}(\varphi'(0)) = 0\]
and
\begin{equation}\label{Equation:anRelation}
f^{(n)}(0) = \frac{1}{\lambda}\sum_{k=1}^{n-1} f^{(k+1)}(0)B_{n,k}\left(0,\varphi''(0),\dots,\varphi^{(n-k+1)}(0)\right)
\end{equation} for $n > 1$.  From \myeqref{Equation:a0Relation}, it must also hold that $f(0) = 0$ as well.
\end{remark}

We present the main result of this section as Theorem \ref{Theorem:SpectralResults}.  We see that all maps of the form $\varphi(z)=az^M$ induce a quasinilpotent operator $D_\varphi$ on $\D$, except for $az^2$.  

\begin{theorem}\label{Theorem:SpectralResults}
If $\varphi(z) = az^M$, for $0 < |a| < 1$ and $M$ in $\N$, then the spectrum of $D_\varphi$ is given by \[\sigma(D_\varphi) = \begin{cases}\hfil \{0, 2a\} & \text{if $M = 2$}\\
\hfil\{0\} & \text{otherwise.}
\end{cases}
\]
\end{theorem}

\noindent For the proof of Theorem \ref{Theorem:SpectralResults}, we require the following lemmas.  These results determine the quantities $f^{(n)}(0)$ if $f$ is an eigenfunction associated with a non-zero eigenvalue $\lambda$.  The cases of $n < M$, $n=M$, and $n > M$ are considered separately.

\begin{lemma}\label{Lemma:CoefficientsLessThanM} 
Suppose $\varphi(z) = az^M$, for $0 < |a| < 1$ and $M \geq 2$.  If $\lambda$ is a non-zero eigenvalue of the induced operator $D_\varphi$ with associated eigenfunction $f$, then $f^{(n)}(0) = 0$ for all $n < M$.
\end{lemma}

\begin{proof}
Note that Remark \ref{Remark:BellPolynomialReduction} shows the lemma to be true for the case $M=2$.  For $M > 2$, the same remark shows $f^{(n)}(0) = 0$ for $n \leq 1$.  Thus we consider the situation that $M > 2$ and $1 < n < M$.

Observe that for $k$ in $\N$ with $1 \leq k \leq n-1$ it follows that $2 \leq n-k+1 \leq n < M$.  Thus $\varphi^{(\ell)}(0) = 0$ for all $1 \leq \ell \leq n-k+1$.  So \myeqref{Equation:anRelation}, in conjunction with Lemma \ref{Lemma:BellPolynomialValues} yields
\[\begin{aligned}
f^{(n)}(0) &= \frac{1}{\lambda}\sum_{k=1}^{n-1}f^{(k+1)}(0)B_{n,k}\left(0,\varphi''(0),\dots,\varphi^{(n-k+1)}(0)\right)\\
&= \frac{1}{\lambda}\sum_{k=1}^{n-1}f^{(k+1)}(0)B_{n,k}(0,\dots,0)\\
&= 0,
\end{aligned}\] as desired.
\end{proof}

\begin{lemma}\label{Lemma:CoefficientEqualToM}
Suppose $\varphi(z) = az^M$, for $0 < |a| < 1$ and $M \geq 2$.  Furthermore, suppose $\lambda$ is a non-zero eigenvalue of the induced operator $D_\varphi$ with associated eigenfunction $f$.  Then $f^{(M)}(0) = 0$ unless $M=2$ and $\lambda=2a$.
\end{lemma}

\begin{proof}
First note that $\varphi^{(\ell)}(0) = 0$ for all $\ell \neq M$ and $\varphi^{(M)}(0) = (M!)a$.  Suppose $k$ in $\N$ such that $2 \leq k \leq M-1$.  Observe that $2 \leq M-k+1 \leq M-1$.  Thus \myeqref{Equation:anRelation} reduces to 
\begin{equation}\label{Equation:MDerivative}
\begin{aligned}
f^{(M)}(0) &= \frac{1}{\lambda}\sum_{k=1}^{M-1} f^{(k+1)}(0)B_{M,k}\left(0,\varphi''(0),\dots,\varphi^{(M-k+1)}(0)\right)\\
&= \frac{1}{\lambda}f''(0)B_{M,1}\left(0,\varphi''(0),\dots,\varphi^{(M)}(0)\right)\\
&\qquad + \frac{1}{\lambda}\sum_{k=2}^{M-1} f^{(k+1)}(0)B_{M,k}\left(0,\varphi''(0),\dots,\varphi^{(M-k+1)}(0)\right)\\
&= \frac{B_{M,1}\left(0,\dots,0,(M!)a\right)}{\lambda}f''(0) + \frac{1}{\lambda}\sum_{k=2}^{M-1} f^{(k+1)}(0)B_{M,k}(0,\dots,0)\\
&= \frac{(M!)a}{\lambda}f''(0).
\end{aligned}
\end{equation} Suppose $M = 2$.  Then \myeqref{Equation:MDerivative} becomes $f''(0) = \frac{2a}{\lambda}f''(0)$.  So it must be that $f''(0) = 0$ unless $\lambda = 2a$.  Finally, suppose $M > 2$.  Then $f''(0) = 0$ by Lemma \ref{Lemma:CoefficientsLessThanM}.  So \myeqref{Equation:MDerivative} reduces to $f^{(M)}(0) = 0$.
\end{proof}

\begin{lemma}\label{Lemma:CoefficientsGreaterThanM} 
Suppose $\varphi(z) = az^M$ for $0 < |a| < 1$ and $M \geq 2$.  If $\lambda$ is a non-zero eigenvalue of the induced operator $D_\varphi$ with associated eigenfunction $f$, then $f^{(n)}(0) = 0$ for all $n > M$.
\end{lemma}

\begin{proof}
We will prove this result by induction on $n$.  First, we will show $f^{(M+1)}(0) = 0$.  We will consider the cases of $M=2$ and $M > 2$ separately.  
\begin{case}{1}
Suppose $M=2$.  We can write \myeqref{Equation:anRelation} as 
\[\begin{aligned}
f'''(0) &= \frac{1}{\lambda}f''(0)B_{3,1}\left(0,\varphi''(0),\varphi'''(0)\right) + \frac{1}{\lambda}f'''(0)B_{3,2}\left(0,\varphi''(0)\right)\\
&= \frac{1}{\lambda}f''(0)B_{3,1}(0,2a,0) + \frac{1}{\lambda}f'''(0)B_{3,2}(0,2a).
\end{aligned}\] From Lemma \ref{Lemma:BellPolynomialValues} it follows that $B_{3,1}(0,2a,0) = 0 = B_{3,2}(0,2a)$.  Thus when $M=2$, $f^{(M+1)}(0) = 0$.

Now, suppose $f^{(3)}(0) = \cdots = f^{(n-1)}(0) = 0$ for some $n \geq 4$.  We will now show $f^{(n)}(0) = 0$.  We write \myeqref{Equation:anRelation} as
\[\begin{aligned}
f^{(n)}(0) &= \frac{1}{\lambda}\sum_{k=1}^{n-1}f^{(k+1)}(0)B_{n,k}\left(0,\varphi''(0),\dots,\varphi^{(n-k+1)}(0)\right)\\
&= \frac{1}{\lambda}f''(0)B_{n,1}(0,2a,0,\dots,0) + \frac{1}{\lambda}f^{(n)}(0)B_{n,n-1}(0,2a)\\
&\qquad + \sum_{k=2}^{n-2}f^{(k+1)}(0)B_{n,k}(0,2a,0\dots,0).
\end{aligned}\] For $2 \leq k \leq n-2$, we have $3 \leq k+1 \leq n-1$.  So $f^{(k+1)}(0) = 0$ by the inductive hypothesis.  From Lemma \ref{Lemma:BellPolynomialValues}, $B_{n,1}(0,2a,0,\dots,0) = 0$ and  $B_{n,n-1}(0,2a) = 0$ since $n > 1$.  Thus $f^{(n)}(0) = 0$.
\end{case}

\begin{case}{2}
Suppose $M>2$.  We will first show $f^{(M+1)}(0) = 0$.  In this case, \myeqref{Equation:anRelation} becomes
\[\begin{aligned}
f^{(M+1)}(0) &= \frac{1}{\lambda}\sum_{k=1}^{M-1}f^{(k+1)}(0)B_{M+1,k}\left(0,\varphi''(0),\dots,\varphi^{(M-k+2)}(0)\right) + \frac{1}{\lambda}f^{(M+1)}(0)B_{M+1,M}(0,0).
\end{aligned}\] For $1\leq k\leq M-1$ it follows that $f^{(k+1)}(0) = 0$ by Lemmas \ref{Lemma:CoefficientsLessThanM} and \ref{Lemma:CoefficientEqualToM} since $2 \leq k+1 \leq M$.  From Lemma \ref{Lemma:BellPolynomialValues} we have $B_{M+1,M}(0,0) = 0$.  Thus $f^{(M+1)}(0) = 0$.

Finally, suppose $f''(0) = \cdots = f^{(n-1)}(0) = 0$ for some $n \geq M+2$.  We will show $f^{(n)}(0) = 0$.  By writing \myeqref{Equation:anRelation} as
\[\begin{aligned}
f^{(n)}(0) &= \frac{1}{\lambda}\sum_{k=1}^{n-1} f^{(k+1)}(0)B_{n,k}\left(0,\varphi'(0),\dots,\varphi^{(n-k+1)}(0)\right)\\
&= \frac{1}{\lambda}\sum_{k=1}^{n-2}f^{(k+1)}(0)B_{n,k}\left(0,\varphi''(0),\dots,\varphi^{(n-k+1)}(0)\right) + \frac{1}{\lambda}f^{(n)}(0)B_{n,n-1}(0,\varphi''(0))\\
&= \frac{1}{\lambda}f^{(n)}(0)B_{n,n-1}(0,0)\\
&= 0.
\end{aligned}\]
\end{case}
\noindent In either case, $f^{(n)}(0) = 0$ for all $n > M$.
\end{proof}

\noindent We can now prove Theorem \ref{Theorem:SpectralResults} as follows.  

\begin{proof}[Proof of Theorem \ref{Theorem:SpectralResults}]
Let $\lambda$ be a non-zero eigenvalue of $D_\varphi$.  As Theorem \ref{Theorem:SpectrumLinear} shows, the conclusion holds for the case of $M=1$. We proceed by considering the following two cases.
\begin{case}{1}
Suppose $M=2$.  Note the function $g(z) = z^2$ is a non-zero function in $\mathcal{D}$.  For each $z$ in $\D$, we see  
\[(D_\varphi g)(z) = g'(\varphi(z)) = 2\varphi(z) = 2(az^2) = 2ag(z).\] Thus $2a$ is an eigenvalue of $D_\varphi$ with associated eigenfunction $g$. So $\{0,2a\} \subseteq \sigma(D_\varphi)$.  Finally, assume, for purposes of contradiction, that $\lambda$ is a non-zero eigenvalue of $D_\varphi$ that is not $2a$, with associated eigenfunction $f$.  Then it follows from Lemmas \ref{Lemma:CoefficientsLessThanM}, \ref{Lemma:CoefficientEqualToM}, and \ref{Lemma:CoefficientsGreaterThanM} that $f^{(n)}(0) = 0$ for all $n \geq 0$.  So $f$ is identically 0, a contradiction.  Thus $\sigma(D_\varphi) = \{0,2a\}$.
\end{case}
\begin{case}{2}
Suppose $M > 2$.  Then it follows directly from Lemmas \ref{Lemma:CoefficientsLessThanM}, \ref{Lemma:CoefficientEqualToM}, and \ref{Lemma:CoefficientsGreaterThanM} that $f^{(n)}(0) = 0$ for all $n \geq 0$.  So $f$ is identically 0, a contradiction.  Thus $\sigma(D_\varphi) = \{0\}$.
\end{case}
\noindent Therefore, the spectrum of $D_\varphi$ has been established, as desired.
\end{proof}

\begin{remark}
The arguments made in this section apply to a functional Banach space $\mathcal{X}$ that contains the polynomials for which $\varphi(z) = az^M$, $0 < |a| < 1$ and $M$ in $\N$, induces a compact operator $D_\varphi$.  As an example, these results apply to $D_\varphi$ acting on $H^2(\D)$ and thus adds to the results in Section 2 of \cite{FatehiHammond:2020}. Additionally, these results hold for $D_\varphi$ acting on the Bloch space, adding to the results of \cite{Ohno:2009}.
\end{remark}

\begin{remark}
At the end of \cite{FatehiHammond:2020}, the authors pose the question: 

\centerline{``Is $D_\varphi$ [acting on $H^2$] quasinilpotent whenever $\varphi$ is univalent and $\|\varphi\|_\infty < 1$?"} 

\noindent While Theorem \ref{Theorem:SpectralResults} does not answer this question, in light of the previous remark it does show that the symbols of quasinilpotent $D_\varphi$ acting on $H^2$ or $\mathcal{D}$ need not be univalent. It also shows that $\|\Ph\|_{\infty}<1$ alone is not enough to guarantee that the operator is quasinilpotent. We believe this line of inquiry poses an interesting open question.
\end{remark}

\bibliographystyle{amsplain}
\bibliography{references.bib}
\end{document}